\documentclass{amsart}

\usepackage{amsmath,amsthm, latexsym, amssymb,mathrsfs,  tikz-cd}
\usepackage{stackengine,scalerel}

\newcommand{\HT}{\mathrm{HT}}

\newcommand{\GL}{\mathrm{GL}}

\newcommand{\mc}[1]{\mathcal{#1}}
\newcommand{\mbb}[1]{\mathbb{#1}}
\newcommand{\mr}[1]{\mathrm{#1}}

\newcommand{\diag}{\mathrm{diag}}

\newtheorem{theorem}{Theorem}

\newtheorem{lemma}{Lemma}

\theoremstyle{definition}

\newtheorem{remark}{Remark}

\begin{document}

\title{Slope classicality via completed cohomology}
\begin{abstract} 
We give a  new proof of the slope classicality theorem in classical and higher  Coleman theory for modular curves at arbitrary level using the completed cohomology classes attached to overconvergent modular forms. The latter give an embedding of the quotient of overconvergent modular forms by classical modular forms, which is the obstruction space for classicality in either cohomological degree, into a unitary representation of $\GL_2(\mbb{Q}_p)$. The $U_p$ operator becomes a double-coset, and unitarity yields the slope vanishing. 
\end{abstract}

\author{Sean Howe}
 
\address{Department of Mathematics, University of Utah, Salt Lake City 84112}
 \email{sean.howe@utah.edu}

\maketitle 

\section{Introduction and proof of slope classicality}
Fix a sufficiently small compact open subgroup $K^p \leq \GL_2(\mbb{A}_f^{(p)})$ and let $\mbb{C}_p$ be the completion of an algebraic closure of $\mbb{Q}_p$. 
Let $X_1(p^n)/\mbb{C}_p$ be the smooth compactification of the modular curve parametrizing elliptic curves with a point of exact order $p^n$ and level $K^p$ structure. Everywhere below we view $X_1(p^n)$ as an adic space over $\mbb{C}_p$. The \emph{closed} canonical ordinary locus $X_{1}(p^n)_e$ is the topological closure of the locus of rank one points parameterizing elliptic curves of ordinary reduction equipped with a point generating the canonical subgroup of level $p^n$. We write $X_{1}(p^n)_w = X_1(p^n)\backslash X_1(p^n)_e$ for its open complement (the subscripts $e$ and $w$ refer to the trivial and non-trivial elements of the Weyl group for $\GL_2$). 

Writing $\omega$ for the modular sheaf, the space $H^0(X_{1}(p^n)_e, \omega^k)$ is naturally identified with the direct sum of spaces of overconvergent modular forms of weights $\kappa$ such that $\kappa=z^k \chi$ for $\chi$ a character of $(\mbb{Z}/p^n\mbb{Z})^\times$. From the perspective of the higher Coleman theory of Boxer-Pilloni \cite{boxer-pilloni:higher-hida-modular, boxer-pilloni:higher-coleman}, it is natural to also consider the compactly supported cohomology $H^1_c(X_1(p^n)_w, \omega^k)$. These groups are related by the exact sequence of compactly supported topological sheaf cohomology 

\begin{equation}\label{eq.csc-seq} 
\begin{tikzcd}
	0 & {H^0(X_1(p^n), \omega^k) } & {H^0(X_1(p^n)_e, \omega^k)} \\
	{H^1_c(X_1(p^n)_w, \omega^k) } & { H^1(X_1(p^n), \omega^k)} & 0.
	\arrow[from=1-1, to=1-2]
	\arrow[from=1-2, to=1-3]
	\arrow[from=1-3, to=2-1]
	\arrow[from=2-1, to=2-2]
	\arrow[from=2-2, to=2-3]
\end{tikzcd}\end{equation}
arising from the exact sequence of sheaves on the topological space $X_1(p^n)$
\[ 0 \rightarrow j_!j^{-1} \omega^k \rightarrow \omega^k \rightarrow i_*i^{-1} \omega^k \rightarrow 0,\; j: X_{1}(p^n)_w \hookrightarrow X_1(p^n),\, i: X_1(p^n)_e \hookrightarrow X_1(p^n). \]
As in \cite{boxer-pilloni:higher-hida-modular, boxer-pilloni:higher-coleman}, there is an operator $U_p$ on each of these spaces induced by a cohomological correspondence and extending a classical double-coset Hecke operator $U_p$ on $H^\bullet(X_1(p^n), \omega^k)$ (up to matching choices of the normalization). 
For any $s \in \mbb{R}$ and $\mbb{C}_p$-vector space $V$ equipped with an action of a linear operator $U_p$, we can pass to the part $V^{<s}$ of slope $<s$, defined to be the span of all generalized eigenspaces of $U_p$ for eigenvalues $\lambda$ with $|\lambda|>p^{-s}$. The main result is:
\begin{theorem}[Slope classicality]\label{theorem.classicality}For $t \in \mbb{Z}\backslash{\{0\}}$,   (\ref{eq.csc-seq}) induces isomorphisms
\begin{align*} H^0(X_1(p^n), \omega^{1+t})^{<|t|} &=H^0(X_1(p^n)_e,\omega^{1+t})^{<|t|} \textrm{ and } \\  H^1_c(X_1(p^n)_w, \omega^{1+t})^{< |t|}&=H^1(X_1(p^n), \omega^{1+t})^{<|t|}. \end{align*}
\end{theorem}

In cohomological degree zero, this is a result of Coleman \cite{coleman:classical,coleman:higher-level}. In degree one, this is a result of Boxer-Pilloni \cite{boxer-pilloni:higher-hida-modular, boxer-pilloni:higher-coleman}\footnote{In \cite{boxer-pilloni:higher-hida-modular} and the introduction of \cite{boxer-pilloni:higher-coleman} this result is stated at level $\Gamma_0(p)$ using the smaller group of cohomology with support in $\overline{X_0(p)_w^{\mr{ord}}}$. It is immediate from the results of loc. cit. that this smaller space has the same finite slope part, and arbitrary level is treated in \cite[Theorem 5.12.2]{boxer-pilloni:higher-coleman} --- we thank George Boxer for answering our questions about these results in the higher level case.} (who also reprove Coleman's result). It is also an immediate consequence of the following lemma, which itself is immediate from the results of \cite{howe:ocmfht} or \cite{pan:locally-an}. To state it, we denote by $X$ the infinite level (compactified) modular curve of prime-to-$p$ level $K^p$. It admits an action of $\GL_2(\mbb{Q}_p)$ and, by Scholze's primitive comparison (see \cite[Corollary 4.4.3]{pan:locally-an}), $H^1(X,\mc{O}_X)$ is identified with the $\mbb{C}_p$-completed cohomology of the tower of modular curves of prime-to-$p$ level $K^p$ --- below we need only that it is a Banach space with a unitary action of $\GL_2(\mbb{Q}_p)$ (it is unitary because the unit ball, i.e. the image of $H^1(X, \mc{O}^+_X)$, is preserved). In the following $N\leq \GL_2$ is the group of upper triangular unipotents.
\begin{lemma}\label{lemma:Up} Let $U_p$ denote the Hecke operator of \cite{boxer-pilloni:higher-hida-modular, boxer-pilloni:higher-coleman} (the normalization depends on the weight; see \S\ref{s.normalizations}). For $t \neq 0$, the cup products of \cite{howe:ocmfht} give an embedding   
\[ H^0(X_1(p^n)_e, \omega^{1+t})/H^0(X_1(p^n),\omega^{1+t}) \hookrightarrow H^1(X, \mc{O}_X)^{N(\mbb{Z}_p)} \]
that matches $U_p$ with the double coset operator $p^{|t|}\cdot[N(\mbb{Z}_p)\diag(p,1)N(\mbb{Z}_p)].$ 
\end{lemma}

\begin{proof}[Proof of Theorem \ref{theorem.classicality}, assuming Lemma \ref{lemma:Up}] Consider the operator
\[ p^{|t|}\cdot[N(\mbb{Z}_p)\diag(p,1)N(\mbb{Z}_p)] = p^{|t|} \sum_{i=0}^{p-1} \begin{bmatrix} p & i \\ 0 & 1 \end{bmatrix} \]
 on $H^1(X, \mc{O}_X)^{N(\mbb{Z}_p)}$. Because the action of $\GL_2(\mbb{Q}_p)$ is unitary, it has operator norm $\leq 1/p^{|t|}$, so the slope $<|t|$ part vanishes in $H^1(X, \mc{O}_X)^{N(\mbb{Z}_p)}$. If we write 
\[ Q_t:=\left(H^0(X_1(p^n)_e, \omega^{1+t})/H^0(X_1(p^n),\omega^{1+t})\right),\]
then, combining the above with Lemma \ref{lemma:Up}, we find that for $t \neq 1$, $Q_t^{<|t|}=0$.  To obtain Theorem \ref{theorem.classicality}, we split (\ref{eq.csc-seq}) for $k=1+t$ into two short exact sequences:
\begin{eqnarray*}
	 0 \rightarrow H^0(X_1(p^n), \omega^{1+t}) \rightarrow H^0(X_1(p^n)_e, \omega^{1+t}) \rightarrow Q_t \rightarrow 0 \textrm{ and }\\ 0 \rightarrow Q_t \rightarrow 
	{H^1_c(X_1(p^n)_w, \omega^{1+t}) } \rightarrow { H^1(X_1(p^n), \omega^{1+t})} \rightarrow 0. 
\end{eqnarray*}
Taking the slope $<|t|$ part yields the isomorphisms -- for the first sequence this is immediate since this functor is always left exact, and for the second sequence right exactness follows from compactness of the $U_p$ operator on overconvergent forms. 
\end{proof}

It thus remains only to prove Lemma \ref{lemma:Up}. This is essentially immediate from the results of \cite{howe:ocmfht} or \cite{pan:locally-an}, \emph{once the $\GL_2(\mbb{Q}_p)$-actions are matched up.} This matching is actually a bit subtle, as there are multiple possible conventions for the Hodge-Tate period map and the equivariant structure on the modular sheaf. Any set of choices gives the same $\GL_2(\mbb{Q}_p)$-action modulo inverse transpose and some determinants, so often the precise choices are irrelevant. Here though we must follow a power of $p$ coming from the action of $\diag(p,1)$, so it is crucial to screw our heads on exactly right on this point. In the next section we fix normalizations then prove Lemma \ref{lemma:Up}. 

\section{Normalizations and proof of Lemma \ref{lemma:Up}}\label{s.normalizations}

\subsection{Choices}
We fix the action of $\GL_2(\mbb{Q}_p)$ on $X$ so that, over the non-compactified infinite level curve $Y$, $\GL_2(\mbb{Q}_p)=\mr{Aut}(\mbb{Q}_p^2)$ acts by composition with the trivialization of the Tate module of the universal elliptic curve -- i.e., we use the action on the homological normalization of the moduli problem. This differs by an inverse transpose from the cohomological normalization, where the action is on the trivialization of the first \'{e}tale cohomology of the universal elliptic curve. 

We take the Hodge-Tate period map $\pi_\HT: X \rightarrow \mbb{P}^1$ so that $\pi_{\HT}|_Y$ is the classifying map for the Hodge-Tate line inside of the first \'{e}tale cohomology of the universal elliptic curve. Thus over $Y$ we have a $\GL_2(\mbb{Q}_p)$-equivariant commuting diagram
\[\begin{tikzcd}
	{\mc{O}_X e_1 \bigoplus \mc{O}_X e_2 } & {\mc{O}_X \otimes V_p(E)} & {\omega_{E^\vee}} \\
	{\mc{O}_{\mbb{P}^1}x \bigoplus \mc{O}_{\mbb{P}^1} y} & {\mc{O}_{\mbb{P}^1} \otimes H^0(\mbb{P}^1,\mc{O}_{\mbb{P}^1}(1)) } & {\mc{O}_{\mbb{P}^1}(1).}
	\arrow["{\pi_\HT^*}", from=2-1, to=1-1]
	\arrow["{\pi_\HT^*}", from=2-2, to=1-2]
	\arrow["\sim", from=1-1, to=1-2]
	\arrow["\sim", from=2-1, to=2-2]
	\arrow["{\pi_\HT^*}", from=2-3, to=1-3]
	\arrow[two heads, from=2-2, to=2-3]
	\arrow[two heads, from=1-2, to=1-3]
\end{tikzcd}\]
Here $e_1$ and $e_2$ are the universal basis for the Tate module $V_p(E)=T_p(E)[1/p]$, $x$ and $y$ are the standard basis for $H^0(\mbb{P}^1, \mc{O}_{\mbb{P}^1}(1))$ so that homogeneous coordinates are $[x:y]$, and $E^\vee$ denotes the dual of the universal elliptic curve. Of course, there is a canonical isomorphism $E^\vee \cong E$ inducing $\omega_{E^\vee} \cong \omega_E$, however, this isomorphism does not respect the natural $\GL_2(\mbb{Q}_p)$-equivariant structures! Equivariantly,\[ \omega_{E^\vee} = \omega_E\otimes |\det|^{-1}, \]
where here $|\det|$ comes from the action of isogenies on $H^1(E, \Omega_E)$. Note that this twist is actually on the entire $\GL_2(\mbb{A}_f)$-action (via $(g_\ell)_\ell \mapsto \prod_\ell |\det(g_\ell)|_\ell^{-1}$), so that the distinction between these equivariant structures also determines the normalization of the prime-to-$p$ Hecke operators. Below we will continue as in the introduction to write simply $\omega$ for the modular sheaf, with the understanding that we have adopted the equivariant structure described above. 
 \newcommand{\cl}{{\mr{cl}}}

Under the natural map to $X \rightarrow X_1(p^n)$, $X_1(p^n)_e$ is the image of $\pi_\HT^{-1}([0:1])$, and the action of $\GL_2(\mbb{Q}_p)$ on $\langle x, y \rangle = H^0(\mbb{P}^1,\mc{O}_{\mbb{P}^1}(1))$ is by the standard representation 
\[ \begin{bmatrix} a & b \\ c & d\end{bmatrix} \cdot x= ax+cy \textrm{ and } \begin{bmatrix} a & b \\ c& d \end{bmatrix} \cdot y= bx+ dy.  \]

\subsection{The $U_p$ operator}
The operator $U_p^\mr{naive}$ at level $X_1(p^n)$ of \cite{boxer-pilloni:higher-hida-modular, boxer-pilloni:higher-coleman} is defined using the correspondence $C$ parameterizing degree $p$-isogenies $\psi:(E_1, P_1) \rightarrow (E_2, P_2)$ (here we suppress prime-to-$p$ level structure from the notation). Writing the two obvious projections as $p_1,p_2 : C\rightarrow X_1(p^n)$, $U_p^\mr{naive}$ is defined on $\omega^k$ as $\mr{tr}\circ {p_{1}}_!\circ \psi^* \circ p_2^*$. Given a geometric point $(E, P)$ that is not a cusp and a non-zero differential $\eta$ on $E$, we can compute $(U^{\mr{naive}}_p f)(E, P, \eta)$ as follows: first choose a basis $(e_1, e_2)$ of $T_p(E)$ such that $e_1$ reduces to $P$ mod $p^n$. Then, for $0 \leq i \leq p-1$, write 
\[ \psi_{i}: E \rightarrow E_i:=E/\langle i\overline{e}_1 + \overline{e}_2\rangle, \]
where $\overline{e}_i$ denotes the image of $e_i$ in $E[p]$. Then $\psi_i^*: \omega_{E_i} \rightarrow \omega_{E}$ is invertible, and 
\[ (U^{\mr{naive}}_p f)(E, P, \eta) = \sum_{i=0}^{p-1} f(E_i, \psi_i(P), (\psi_i^*)^{-1}\eta). \]

 We will now realize this same $U_p^\mr{naive}$ as a double-coset operator: let $B$ denote the upper triangular Borel in $\GL_2$. The space of overconvergent modular forms of weight $k$ at any finite level $\Gamma_1(p^n)$ is naturally embedded as the $B_1(p^n)=\Gamma_1(p^n) \cap B(\mbb{Q}_p)$ invariants in the $B(\mbb{Q}_p)$-representation
\[ M_k^\dagger := H^0([0:1], ({\pi_\HT}_* \pi_\HT^* \mc{O}(k))^{\mr{sm}}), \]
where here the superscript $\mr{sm}$ denotes the subsheaf of ${\pi_\HT}_* \pi_\HT^* \mc{O}(k)={\pi_{\HT}}_*\omega^k$ whose sections over a quasi-compact open $U$ are those with locally constant orbit map for the action the open stabilizer of $U$ in $\GL_2(\mbb{Q}_p)$ (i.e. the sections coming from finite level). For more on this construction, see \cite[\S3.1]{howe:ocmfht}.

The space $M_k^\dagger$ contains the space of classical modular forms 
\[ M_k^\cl = H^0(\mbb{P}^1, ({\pi_\HT}_* \pi_\HT^* \mc{O}(k))^{\mr{sm}}) \]
$B(\mbb{Q}_p)$-equivariantly by restriction. 
By a classical computation, the action of 
\[ N(\mbb{Z}_p)\diag(p,1)N(\mbb{Z}_p) = \sum_{i=0}^{p-1} \begin{bmatrix} p & i \\ 0 & 1 \end{bmatrix} =: \sum_{i=0}^{p-1} M_i   \]
on $M_k^{\dagger, N(\mbb{Z}_p)}$ is identified with $U_p^{\mr{naive}}$ on $M_k^{\dagger, B_1(p^n)}=H^0(X_1(p^n), \omega^k)$ --- for completeness, we explain it here. It suffices to check at geometric points away from cusps, so we can compare with the explicit formula given above for $U_p^{\mr{naive}}$. We have
\begin{align*} (N(\mbb{Z}_p)\diag(p,1)N(\mbb{Z}_p) \cdot f)(E, e_1, e_2, \eta)&  = \sum_{i=0}^{p-1} \left(M_i \cdot f\right)(E, e_1, e_2, \eta) \\
& = \sum_{i=0}^{p-1} f(E, p e_1, i e_1 + e_2, \eta).\end{align*}
By the following commuting diagram, 
\[\begin{tikzcd}
	{\mbb{Q}_p^2} & {\mbb{Q}_p^2} & {V_p E} \\
	&& {V_p(E_i)}
	\arrow["{e_1, e_2}", from=1-2, to=1-3]
	\arrow["{ M_i}", from=1-1, to=1-2]
	\arrow["{\psi_i(e_1), \psi_i(\frac{1}{p} (ie_1 + e_2))}"', from=1-1, to=2-3]
	\arrow["{\psi_i^\vee}"', from=2-3, to=1-3]
\end{tikzcd}\]
We see that $f(E, p e_1, i e_1 + e_2, \eta)$ is equal to $f(E_i, \psi_i(e_1), \psi_i(\frac{1}{p} (ie_1 + e_2)), \eta_i)$ for some $\eta_i$. One is tempted to guess that $\eta_i$ is $(\psi_i^\vee)^* \eta$, but this is not the case! Indeed, this would be the true only if we used the $\omega_E$-equivariant structure. Since we use the $\omega_{E^\vee}$ equivariant structure, we must replace $\psi_i^\vee$ with its dual $\psi_i$ and then $\eta_i=(\psi_i^{*})^{-1}\eta$. Thus we do indeed recover the formula described above for $U_p^{\mr{naive}}$. 

Recall \cite[sentence preceding Theorem 5.13]{boxer-pilloni:higher-hida-modular} the normalized operator\[ U_p := \begin{cases} p^{-1} U_p^{\mr{naive}} & \textrm{ if } k \geq 1\\  	
p^{-k} U_p^{\mr{naive}} & \textrm{ if } k \leq 1 
 \end{cases} \textrm{ so that } U_p^{\mr{naive}} = \begin{cases} p U_p & \textrm{ if } k \geq 1\\  	
p^{k} U_p& \textrm{ if } k \leq 1.  \end{cases} \]

\subsection{Proof of Lemma \ref{lemma:Up} and concluding remarks}
\begin{proof}[Proof of Lemma \ref{lemma:Up}]
We first treat the case $k=1+t \geq 2$. Then, for any $s \in M_k^\dagger$, $s/x^k$ is a function on $\mc{O}_X$ defined on the pre-image of a punctured neighborhood of $[0:1]$ under $\pi_\HT$. It thus determines a Cech cohomology class $[s/x^k]$ in $H^1(X,\mc{O}_X)$. Then \cite[Theorem A]{howe:ocmfht} implies that $s \mapsto [s/x^k]$ induces an injection $M_k^\dagger/M_k^{\mr{cl}} \hookrightarrow H^1(X, \mc{O}_X)$ -- actually, in \cite{howe:ocmfht} the results are stated using cusp forms and compactly supported completed cohomology, but given the identification of $H^1(X,\mc{O}_X)$ with completed cohomology, one obtains the desired statement by the same arguments. The map is $B(\mbb{Q}_p)$-equivariant if one twists the action on $M_k^\dagger$ by 
\[ \begin{bmatrix} a & b \\ 0 & c \end{bmatrix} \mapsto a^{-k} \]
(the twist comes from the action on $x^{k}$, of course!). We deduce that $p^{-k}U_p^{\mr{naive}}=p^{-k}(pU_p)=p^{-t}U_p$ is identified with $[N(\mbb{Z}_p)\diag(p,1)N(\mbb{Z}_p)]$, as desired. 

We now treat the case $k=1+t \leq 0$. In this case, \cite[Theorem A]{howe:ocmfht} show that $s \mapsto [s/(xy^t)]$ induces an embedding $M_k^\dagger/M_k^\cl \hookrightarrow H^1(X, \mc{O}_X)$. Actually, here one must be slightly more careful invoking the arguments of \cite{howe:ocmfht}, which are stated with cusp forms, in the case $k=0$: of course $M_k^\cl= 0$ when $k < 0$, but when $k=0$ we have that $M_0^\cl$ is the locally constant functions whereas the cusp forms are still trivial. However, it is elementary to see that $M_0^\cl$ is in the kernel (as $s/(xy^{-1})=s/z$ extends to a function on the complement of $[0:1]$ where $1/z$ is a local coordinate), and the argument of loc. cit. still establishes an injection on the quotient by $M_0^\cl$. The embedding is $B(\mbb{Q}_p)$-equivariant if we twist  $M_k^\dagger$ by 
\[ \begin{bmatrix} a & b \\ 0 & c \end{bmatrix} \mapsto a^{-1}c^{-t}, \]
where again the twist comes from the action on $xy^t$. We deduce that $p^{-1}U_p^\mr{naive}=p^{-1}(p^kU_p)=p^{t}U_p$ is matched with $[N(\mbb{Z}_p)\diag(p,1)N(\mbb{Z}_p)]$, concluding the proof. 
\end{proof}

\begin{remark}
Lemma \ref{lemma:Up} can also be deduced from \cite[Theorem 1.0.1]{pan:locally-an}, and this has the advantage that loc. cit. is stated already with completed cohomology instead of compactly supported completed cohomology and cusp forms. We have used \cite{howe:ocmfht} above simply because it was easier for us to check carefully our own normalizations! 
\end{remark}

\begin{remark}
The vectors used for $k \leq 0$ also exist for $k \geq 2$, where they induce an injection on all overconvergent modular forms. The same argument then recovers the fact that $U_p$ has non-negative slopes when $k \geq 2$ (of course, it is much simpler to deduce this from the action on $q$-expansions!). Representation-theoretically, this vector comes from a highest weight Verma module, which, when $k \geq 2$, admits an algebraic quotient; the classical forms are exactly those that factor through the algebraic quotient, and the vector we used above for the $k \geq 2$ case is the lower highest weight vector generating the kernel. This perspective is explained in \cite{howe:ocmfht}. 
\end{remark}

\begin{remark}
The reason one is led to use different normalizations depending on $k \leq 0$ or $k \geq 2$  is mostly explained by the form of the Hodge-Tate sequence  
\[ \mr{Lie} E(1)\rightarrow T_p(E) \otimes \mc{O}_Y \rightarrow \omega_{E^\vee}. \]
Indeed, since we are using the double-coset operator for the $p$-integral matrix $\diag(p,1)$  acting on $V_p E$, to obtain an operator with non-negative slopes it is natural to use the equivariant structure from $\omega_{E^\vee}$ for $k = 1$ and from $\mr{Lie}E=\omega_{E}^{-1}$ for $k=-1$, and similarly for larger $|k|$ by taking symmetric powers of $T_p(E)$. The equivariant structures differ by an absolute value of the determinant, which mainfests here as different powers of $p$ for the double-coset operator coming from $\diag(p,1)$. 
\end{remark}

\bibliographystyle{plain}
\bibliography{refs}

\end{document}